\documentclass[11pt,letterpaper]{amsart}
\usepackage{mathrsfs}

\usepackage{color}
\usepackage{bm}
\usepackage{amsfonts,amssymb}
\usepackage{dsfont}
\usepackage{amscd}
\usepackage{extarrows}
\usepackage{amsmath}
\usepackage{mathrsfs}
\usepackage{enumerate}
\usepackage{amscd}
\usepackage[all]{xy}
\usepackage{hyperref}
\numberwithin{equation}{section}
\theoremstyle{plain} 

\theoremstyle{definition} 

\newtheorem{thm}{Theorem}[section]
\newtheorem{cor}[thm]{Corollary}
\newtheorem{lem}[thm]{Lemma}

\theoremstyle{definition}
\newtheorem{defn}{Definition}[section]

\theoremstyle{remark}
\newtheorem{rem}{Remark}[section]

\newcommand{\be}{\begin{equation}}
	\newcommand{\ee}{\end{equation}}
\newcommand{\bea}{\begin{eqnarray}}
	\newcommand{\eea}{\end{eqnarray}}
\newcommand{\ben}{\begin{eqnarray*}}
	\newcommand{\een}{\end{eqnarray*}}
\newcommand{\bt}{\begin{split}}
	\newcommand{\et}{\end{split}}
\newcommand{\bet}{\begin{equation}}
	\newcommand{\mc}{\mathbb{C}}
	\newcommand{\mr}{\mathbb{R}}

	\newcommand{\ra}{\rightarrow}
	\newcommand{\beq}{\begin{equation*}}
		\newcommand{\eeq}{\end{equation*}}
	\newcommand{\bi}{\begin{itemize}}
		\newcommand{\ei}{\end{itemize}}

\newcommand{\mo}{\mathcal{O}}

\newcommand{\rw}{\rightarrow}
\newcommand{\rwo}{\mapsto}

\newcommand{\dbar}{\bar{\partial}}
  \newcommand{\im}{\operatorname{Im}}
    \newcommand{\re}{\operatorname{Re}}

\begin{document}
		
\title[A bridge connecting convex analysis and complex analysis]
{A bridge connecting convex analysis and complex analysis and $L^2$-estimate of $d$ and $\bar\partial$}

		\author[F. Deng]{Fusheng Deng}
		\address{Fusheng Deng: \ School of Mathematical Sciences, University of Chinese Academy of Sciences, Beijing 100049, P. R. China}
		\email{fshdeng@ucas.ac.cn}
		
		\author[J. Hu]{Jinjin Hu}
		\address{Jinjin Hu: \ Department of Mathematics and Yau Mathematical Sciences
Center, Tsinghua University, Beijing 100084, P. R. China}
		\email{hujinjin@mail.tsinghua.edu.cn}

        \author[W. Jiang]{Weiwen Jiang}
        \address{Weiwen Jiang: \ Institute of Mathematics, Academy of Mathematics and Systems Science, Chinese Academy of Sciences, Beijing 100190, P. R. China}
        \email{jiangweiwen@amss.ac.cn}
		
		\author[X. Qin]{Xiangsen Qin}
		\address{Xiangsen Qin: \ School of Mathematical Sciences, University of Chinese Academy of Sciences, Beijing 100049, P. R. China}
		\email{qinxiangsen19@mails.ucas.ac.cn}

\begin{abstract}
  We propose a way to connect complex analysis and convex analysis. As applications,
  we derive some results about $L^2$-estimate for $d$-equation and prove some curvature positivity related to convex analysis from well known $L^2$-estimate for $\bar\partial$-equation
        or the results we prove in complex analysis.
 \end{abstract}

		\maketitle
		
\tableofcontents
\section{Introduction}
It is well known that complex analysis and convex analysis are related or analogous in some aspects.
For example, a function $\varphi(z)$ on $\mc^n$ that is independent of the imaginary part of $z$ is plurisubharmonic
if and only if it is convex when viewed as a function on $\mr^n$, and a tube domain of the form $D + i \mr^n\subset\mc^n$
is pseudoconvex if and only if its base $D$ is a convex domain in $\mr^n$.
In the present paper, we propose another principle that connecting complex analysis and  convex analysis in a certain way.
As applications, we show that some important (both old and new) results in convex analysis can be deduced from  analogous results in complex analysis.

The principle can be described as follows.
Starting from a domain (connected open set) $D\subset\mr^n$, one can form a tube domain $T_D:=D+i \mr^n$ in $\mc^n$.
Now the domain $T_D$ admits the obvious symmetry $\mr^n$ given by translations along the imaginary part.
But the trouble is that it is difficult to make use of the symmetry since $\mr^n$ is noncompact as a Lie group.
        So we take a step forward by considering the domain
        $$R_D:=\exp T_D=\{(e^{z_1},\cdots,e^{z_n})|\ (z_1,\cdots,z_n)\in T_D\},$$
        which is a Reinhardt domain in $\mc^n$ that has no intersection with the coordinate axes.
        Now the key point is that translation symmetry of $T_D$ induces rotation symmetry of $R_D$ given by the compact Lie group $\mathbb{T}^n$, the torus group of dimension $n$. Through this way, we get a one to one correspondence between domains in $\mr^n$ and Reinhardt domains in $(\mc^*)^n$.
        It follows that, in principle, problems about $D$ can be translated to analogous rotation-invariant problems about $R_D$, and vice versa.

        This principle has been implicitly used in \cite{DHJ20} and  \cite{DZZ17}  for the study
        of positivity of direct image sheaves, and has applied in \cite{DHQ}, \cite{DJQ} deducing
        some important results about real analysis from analogous results from complex analysis.

        The above principle has a lot of other applications.
        In this paper, we give some typical examples to show the power of it.

        The first example is the following

         \begin{thm}\label{thm:Ber}
         Let $D\subset \mr_x^n$ be a convex domain, $V\subset \mc_\tau^m$ be a pseudoconvex domain,
         and let $W_z:=D+i \mr_y^n\subset \mc_z^n.$ Let $\varphi(\tau,z),\ \psi(\tau, z)$ be plurisubharmonic functions on $V\times W,$ which are independent of $\im(z)$
         (the imaginary part of $z$) for any $z\in W$. We assume that $\psi$ is $C^2$ and is strictly plurisubharmonic with respect to $\tau$.
         Then for any $\bar{\partial}$-closed measurable $(0,1)$-form $f:=\sum^n_{j=1}f_jd\bar{\tau}_j$ in $V$, we can solve the equation $\bar{\partial}u=f$ in $V$ with the estimate
         $$
         \int_{V\times D}|u|^2e^{-\varphi-\psi}d\lambda\leq \int_{V\times D}\sum^n_{j,k=1}\psi^{j\bar k}f_j\bar f_ke^{-\varphi-\psi}d\lambda,
         $$
         provided the right hand side is finite, where $\left(\psi^{j\bar k}\right)_{n\times n}:=\left(\frac{\partial^2\psi}{\partial \tau_j\partial \bar{\tau}_k}\right)^{-1}_{n\times n}$.
         \end{thm}

        Theorem \ref{thm:Ber} is a slight generalization of  \cite[Proposition 2.3]{Ber98}, which plays an essential role in Berndtsson's proof of a Pr\'ekopa-type inequality for plurisubharmonic functions on tube domains.  Berndtsson's original proof involves taking limit of balls in $\mr^n$ and an infinite dimensional version of the Arzel\`a-Ascoli Theorem.
        Our proof of Theorem \ref{thm:Ber}, based on the principle mentioned above, will be very simple and different. Following Berndtsson's method, Inayama proved an analogue of Theorem \ref{thm:Ber} for $L^2$-extension of holomorphic functions.
        \begin{thm}[{\cite[Theorem 1.1]{I21}}]\label{thm:optimal}
          For any $r>0,$ let $\Delta:=\{\tau\in \mc|\ |\tau|<r\}.$ Let $D\subset \mr_x^n$ be a bounded convex domain,
          $\varphi(\tau,z)$ be a plurisubharmonic function on $\Delta_\tau\times (D_x+i \mr^n_y),$ which is independent of $\im(z)$.
          Then there exists a holomorphic function $f$ on $\Delta$ satisfying $f(0)=1$ and
        $$
        \int_{\Delta\times D}|f(\tau)|^2e^{-\varphi(\tau,x)}d\lambda(\tau,x)\leq \pi r^2\int_D e^{-\varphi(0,x)}d\lambda(x),
        $$
        provided the right hand side is finite.
        \end{thm}

        We will show that, combing the optimal $L^2$-extension theorems for holomorphic functions in \cite{Blo13}, \cite{GZ15},
        Theorem \ref{thm:optimal} can be also deduced from the above principle directly.
        Of course, the disc $\Delta$ in Theorem \ref{thm:optimal} can be replaced by general bounded domains in $\mc^k$,
        with the form of the estimate modified accordingly. 
        But to clarify the main idea of the proof, we only consider the case $k=1$.

        We now turn to the discussion of $L^2$-estimate for the $d$-operator on convex domains.
        In complex analysis, the $L^2$-estimate of $\bar\partial$-equation on general pseudoconvex domains
        was established by H\"ormander in 1965 in his fundamental work \cite{Hor65}.
        H\"ormander's result was generalized to much more general context by Demailly and other authors in the early eighties (see e.g. \cite{Dem}).
        Surprisingly, one had not seen analogous work about the $L^2$-estimate for the $d$-operator
        until a decade later when Brascamp and Lieb proved a parallel result in \cite{BL76} on $\mr^n$.
        The $L^2$-estimate for the $d$-operator on general convex domains was proved only recently in \cite{JLY14} for line bundles.
        For positively curved vector bundles on the whole $\mr^n$,
        the $L^2$-estimate for the $d$-operator parallel to Brascamp-Lieb's result is given in \cite{Cor19} in 2019.
        The main motivation to \cite{BL76}, \cite{Cor19} is to prove the Brunn-Minkowski type result in convex analysis for Riemannian vector bundles with positive curvature.
        As indicated in H\"ormander's work, generalizing an estimate for the $d$-operator from the whole $\mr^n$ to general convex domains
        may require a highly nontrivial effort to prove some integral inequalities for certain forms that are not smooth.
        In the present paper, we generalize the result in \cite{Cor19} from $\mr^n$ to general convex domains very easily
        by applying well known results about the $L^2$-estimate of $\bar\partial$ and the above proposed principle.

        \begin{thm}\label{thm:Cor}
        Let $D\subset \mr^n$ be a convex domain, and let $(E,g)$ be a Nakano positive Riemannian vector bundle of finite rank $r$ over $D$, then for any $d$-closed $f\in L^2(D,\Lambda^1 T^*D\otimes E),$
            there exists $u\in L^{2}(D,E)$ with $du=f$ and satisfying the estimate
            $$
            \int_{D}|u|_g^2d\lambda\leq\int_{D}\langle(\Theta^{(E,g)})^{-1}f,f\rangle_{g}d\lambda.
            $$
        \end{thm}

In the above theorem,
$$\Theta^{(E,g)}:=\left(\theta_{jk}^{(E,g)}\right)_{n\times n},\ \theta_{jk}^{(E,g)}:=-\frac{\partial}{\partial x^k}\left(g^{-1}\frac{\partial g}{\partial x^j}\right)$$
 is the curvature operator of $(E,g)$.
 We will give the definition of the involved Nakano positivity of $(E,g)$ later (see Definition \ref{defn:curvature operator}).

        Theorem \ref{thm:Cor} can be generalized to closed $E$-valued forms of higher degree via the same argument.
        We omit the exact formulation for the general context here.  The main idea in the proof of Theorem \ref{thm:Cor} is as follows.
        We transform from $D$ to $R_D$ and then solve the analogous $\bar\partial$-equation on $R_D$ with the $L^2$-estimate known in complex analysis.
        Since everything on $R_D$ that we are handling is invariant under the natural torus action on $R_D$,
        the uniqueness of the minimal solution implies that the minimal solution of the $\bar\partial$-equation is also
        invariant under the torus action. It follows that the minimal solution on $R_D$ induces a solution of the $d$-equation on $D$
        with the desired $L^2$-estimate.

 Now we consider the curvature strict positivity of the direct image bundles associated to a strictly pseudoconvex family of bounded domains (see Definition \ref{def:strict p.s.c family} for definition).
 We consider the families that all fibers are circular domains or Reinhardt domains.
  For any $k\in\mathbb{N}$, let $\rm \rm P^k$ be the space of homogenous polynomials on $\mc^m$ of degree $k$.


 \begin{thm}\label{thm(extend):curvartue positive circular domains}
 Let $\Omega\subset U\times\mc^m$ be a strictly pseudoconvex family of bounded domains over $U\subset\mc^n$,
 and let $(F,h^F)$ be a Nakano semi-positive trivial Hermitian vector bundle of finite rank $r$
 over some neighborhood of $\overline\Omega$ in $U\times\mc^m$.
 We assume that all fibers $\Omega_t\ (t\in U)$ are (connected) circular domains in $\mc^m$ containing the origin and $h^F_{(t,z)}$ is $\mathbb{S}^1$-invariant with respect to $z$.
 Let $\{e_1, \cdots , e_r\}$ be the canonical holomorphic frame of $F$.
 For any $k\in\mathbb{N},\ t\in U$, set
        	$$E^k_t:=\{f=\sum_{\lambda=1}^r f_\lambda \otimes e_\lambda|\ f_\lambda \in \rm \rm P^k\text{ for all }\lambda\}$$
        	with inner product $h_t$ given by
        	$$h_t(f,g):=\int_{\Omega_t} \sum_{\lambda,\mu=1}^rf_\lambda(z)\overline{g_\mu(z)}h^F_{(t,z)}(e_\lambda,e_\mu)d\lambda(z),\ \forall f,g \in E^k_t.$$
        	We set $E^k:=\cup_{t\in U}E^k_t$ and view it as a (trivial) holomorphic vector bundle over $U$ in a natural way, then the curvature of $(E^k,h)$
        	is strictly positive in the sense of Nakano.
        \end{thm}

Considering strictly pseudoconvex family of Reinhardt domains, we have the following parallel result, whose proof is similar to that
of Theorem \ref{thm(extend):curvartue positive circular domains}, so we omit it.

        \begin{thm}\label{qin:application}
        Let $\Omega\subset U\times\mc^m$ be a strictly pseudoconvex family of bounded domains over a domain $U\subset\mc^n,$
        and let $(F,h^F)$ be a Nakano positive trivial Hermitian vector bundle of finite rank $r$ defined on some neighborhood of $\overline\Omega$.
		We assume that all fibers $\Omega_t\ (t\in U)$ are (connected) Reinhardt domains and $h^F_{(t,z)}$ is $\mathbb{T}^n$-invariant with respect to $z$.
        Let $\{e_1, \cdots , e_r\}$ be the canonical holomorphic frame of $F$.
        For any  $k_1,\cdots, k_m\in\mathbb{N},\ t\in U$, set
        $$E^{k_1,\cdots, k_m}_t:=\{f=\sum_{\lambda=1}^r f_\lambda \otimes e_\lambda|\ f_\lambda \in \mc z_1^{k_1}\cdots z_m^{k_m}\text{ for all }\lambda\}$$
        with inner product $h_t$ given by
        $$h_t(f,g):=\int_{\Omega_t} \sum_{\lambda,\mu=1}^rf_\lambda(z)\overline{g_\mu(z)}h^F_{(t,z)}(e_\lambda,e_\mu)d\lambda(z),\ \forall f,g \in E^{k_1,\cdots, k_m}_t.$$
        We set $$E^{k_1,\cdots, k_m}:=\mathop{\cup}_{t\in U}E^{k_1,\cdots, k_m}_t$$ and view it as a (trivial) holomorphic vector bundle over $U$ in a natural way, then the curvature of $(E^{k_1,\cdots, k_m},h)$
        is strictly positive in the sense of Nakano.
        \end{thm}

In Theorem \ref{thm(extend):curvartue positive circular domains} and Theorem \ref{qin:application},
the vector bundle $F$ is assumed to be trivial just for simplicity of presentation.

We should remark that the strict positivity of the curvature of the direct images in Theorem \ref{thm(extend):curvartue positive circular domains} and Theorem \ref{qin:application}
comes from the strict pseudoconvexity of the families,
but not from the curvature positivity of $(F, h^F)$.
Similar considerations have appeared in \cite{Wang17} and \cite{DHQ} for direct images of Hermitian line bundles.
The proofs of Theorem \ref{thm(extend):curvartue positive circular domains} and Theorem \ref{qin:application} rely on a development of the method in \cite{DHQ}.
Our argument depends on Deng-Ning-Wang-Zhou's integral characterization of Nakano positivity in \cite{DNWZ}
and is different from the method in  \cite{Wang17}.


If we just consider curvature positivity but not strict positivity of the direct image,
Theorem \ref{thm(extend):curvartue positive circular domains} and Theorem \ref{qin:application} are due to Deng-Zhang-Zhou \cite{DZZ17} for the line bundle case,
based on the fundamental result of Berndtsson in \cite{Ber09}.

Using our principle and Theorem \ref{qin:application}, we can get the strict curvature positivity of the direct images of Riemannian vector bundles over a strict convex family of
        bounded domains in $\mr^m$ (see Definition \ref{def:strict convex family} for definition).	
        
Recall that a $C^2$ function defined on a domain in $\mr^n$ is called strictly convex if its real Hessian is positive definite everywhere.

\begin{thm}\label{qin:convex}
			 Let $D\subset U_0\times\mr^m$ be a  strictly convex family of bounded domains (with connected fibers) over a domain $U_0\subset\mr^n,$
            and let $(F,g^F)$ be a trivial vector bundle of finite rank $r$ defined on some neighborhood  of $\overline D$.
            Let $\{e_1,\cdots,e_r\}$ be the canonical frame of $F$.
            For any $t\in U_0$, set $E_t:=\mr^r$, with an inner product $g_t^E$ given by
            $$g_t^E(u,v):=\int_{D_t}\sum_{\lambda,\mu=1}^ru_\lambda v_\mu g^F_{(t,x)}(e_\lambda,e_\mu)d\lambda(x)$$
            for all
            $$u:=(u_1,\cdots,u_r)\in E_t,\ v:=(v_1,\cdots,v_r)\in E_t.$$
            We set $E:=\cup_{t\in U_0}E_t$ and view it as a Riemannian (trivial) vector bundle over $U_0$, then the curvature of $(E,g^E)$
            is strictly positive in the sense of Nakano if $(F,g^F)$ is Nakano positive on some neighborhood of $\overline D$.
\end{thm}
		
The curvature positivity of $(E, g^E)$  is originally proved in \cite{Rau13} in the case that $D=\mr^n\times\mr^m$ is a product,
based on the idea of \cite{Ber09} and Fourier transform, and is proved by different methods in \cite{Cor19} and \cite{DHJ20}.
Here we are interested in the curvature strict positivity of $(E, g^E)$ and on deriving it from Theorem \ref{qin:application}
following the principle proposed in the beginning of this section.

\subsection*{Acknowledgements}
The authors thank the referee for suggestion on improving the presentation the article.
This research is supported by National Key R\&D Program of China (No. 2021YFA1003100),
NSFC grants (No. 12471079), and the Fundamental Research Funds for the Central Universities.

		\section{Preliminaries}\label{sec:preliminary}
		In this section, we fix some notations, conventions and  collect some knowledge that are needed in our discussions.
        \subsection{Notations and conventions}\ \\
        \indent Our convention for $\mathbb{N}$ is that $\mathbb{N}:=\{0,1,2,3,\cdots\}$,
        and set $\mc^*:=\{z\in\mc|\ z\neq 0\}.$ For any complex number $z\in \mc$, let $\im(z)$ denote the imaginary part of $z$.
       We say $U\subset\mr^n$ is a domain of $\mr^n$ if it is a connected
        open subset of $\mr^n$. \\
        \indent   For any $k\in\mathbb{N}\cup\{\infty\}$.  If $U\subset \mr^n$ is an open subset, then we let $C^k(U)$ be the space of complex valued functions which are of class $C^k$.
         A real valued function $\varphi\in C^2(U)$ is strictly convex if its real Hessian $\left(\varphi_{jk}\right)_{1\leq j,k\leq n}$ is positive definite for any $x\in U$, where
        $$\varphi_{jk}:=\frac{\partial^2\varphi}{\partial x_j\partial x_k},\ 1\leq j,k\leq n,$$
        and we let $\left(\varphi^{jk}\right)_{1\leq j,k\leq n}$ be the inverse matrix of its real Hessian.
        We recall that an open subset $U\subset\mr^n$ is strictly convex if it has $C^2$-boundary, and it has a $C^2$-boundary defining function which is strictly convex in a neighborhood of
        $\overline{U}$.\\
        \indent If $U\subset\mc^n$ is a open subset, let $\mo(U)$ denote the space of holomorphic functions in $U$.
        An upper semi-continuous function $f\colon U\rw [-\infty,\infty)$
        is plurisubharmonic if its restriction to every complex line in $U$ is subharmonic, i.e. satisfying the submean value inequality.
        A real valued function $\varphi\in C^2(U)$ is strictly plurisubharmonic if its complex Hessian $\left(\varphi_{j\bar{k}}\right)_{1\leq j,k\leq n}$ is positive definite for any $z\in U$, where
        $$\varphi_{j\bar{k}}:=\frac{\partial^2\varphi}{\partial z_j\partial \bar z_k},\ 1\leq j,k\leq n,$$
        and we let $\left(\varphi^{j\bar{k}}\right)_{1\leq j,k\leq n}$ be the inverse matrix of its complex Hessian.
        $U$ is pseudoconvex if it has a smooth plurisubharmonic function $p$ such that its sublevel set $\{z\in U|\ p(z)\leq c\}$ is relatively compact in $U$
        for any $c\in\mr$.  $U$ is strictly pseudoconvex if it has $C^2$-boundary, and it has a $C^2$-boundary defining function which is strictly plurisubharmonic in
        a neighborhood of $\overline{U}$.\\
        \indent Let $d\lambda$ denote the Lebesgue measure on $\mr^n$. Note that all measurable sets, measurable maps, measurable differential forms
        will be Lebesgue measurable. If $A\subset\mr^n$ is a measurable subset, we also write $|A|$ for the measure of
        $A.$ \\
         \indent For any open subset $U\subset\mr^n$, let $L^2(U)$ denote the space of $L^2$-integrable complex valued functions on $U$. For any $f\in L^2(U)$, let $\|f\|_{L^2(U)}$ denote the $L^2$-norm of $f$.  \\
         \indent For any open subset $U\subset\mr^n$ (resp. $U\subset \mc^n$), and for any $k\in\mathbb{N}$
         (resp. $p,q\in\mathbb{N}$), let $\Lambda^k T^*U$ (resp. $\Lambda^{p,q}T^*U)$ denote the bundle of smooth, complex valued $k$-forms (resp. $(p,q)$-forms) on $U$.\\
         \indent If $(X,\omega)$ is a K\"ahler manifold, then we write $d\lambda(\omega)$ for the Riemannian volume form determined by $\omega$.\\
         \indent  A metric on a vector bundle will be of class $C^2$, all sections of a vector bundle will be global unless stated otherwise. Let $(E,h^E)$ be a Riemannian smooth (resp. Hermitian holomorphic) vector bundle of finite rank $r$  over an open subset  $U\subset\mr^n$ (resp. $U\subset\mc^n$),
         and let  $e_1,\cdots,e_r$ be a local smooth (resp, holomorphic) frame of $E$.  For any $x\in U$, let $h^E_x$ denote the inner product on the fiber $E_x$ of $x$.
         For any $1\leq \alpha,\beta\leq r,$ set
         $$h^E_{\alpha\bar{\beta}}:=h^E(e_\alpha,e_\beta),$$
         and let $\left((h^E)^{\alpha\bar{\beta}}\right)_{1\leq\alpha,\beta\leq n}$ be the inverse matrix of
         $\left(h^E_{\alpha\bar{\beta}}\right)_{1\leq \alpha,\beta\leq r}.$
         Let $u,v$ be any measurable sections of $E$,  write
         $$u=\sum_{\lambda=1}^ru_\lambda e_\lambda,\ v=\sum_{\mu=1}^rv_\mu e_{\mu},$$
         then we let
         $$|u|^2:=\sum_{\lambda=1}^r|u_\lambda|^2,\ \langle u,v\rangle_{h^E}:=h^E(u,v)=\sum_{\lambda,\mu=1}^ru_\lambda\overline{v_\mu}h^E(e_\lambda,e_\mu),$$
         $$|u|_{h}^2:=\langle u,u\rangle_{h},\ (u,v)_{h}:=\int_{U}\langle u,v\rangle_{h^E} d\lambda,\ \|u\|_{h^E}^2:=(u,u)_{h^E}.$$
         Let $L^2(U,E)$ be the space of measurable $E$-valued sections which are $L^2$-integrable on $U$, i.e.
         $$L^2(U,E):=\{u\text{ is a measurable } E\text{-valued section of } E|\ \|u\|_{h^E}^2<\infty\}.$$
         For any $k\in\mathbb{N}\cup\{\infty\}$, let $C^k(U,E)$ be the space of $E$-valued sections which are of class $C^k$.
         We also let $C_c^k(U,E)$ be the space of elements of $C^k(U,E)$ which has compact support.
         \subsection{Regularized max function}\ \\
		\indent Let $\psi \in C^{\infty}(\mr)$ be a nonnegative even function with support in $[-1, 1]$ such that
		   $$\int_{\mr} \psi(x) dx=1.$$
		\begin{lem}[{\cite[Lemma 5.18, Chapter I]{Dem}}]\label{lem: Regularization max functions}
			For any $\eta:=(\eta_1, \eta_2) \in (0, +\infty) \times (0, +\infty)$, the function
			$\max\nolimits_{\eta}\colon \mr^2\rw \mr$ defined by
			$$(t_1,t_2)\rwo \int_{\mr^2} \max\{t_1+x_1, t_2+x_2\} \frac{1}{\eta_1 \eta_2} \psi\left(\frac{x_1}{\eta_1}\right) \psi\left(\frac{x_2}{\eta_2}\right) d x_1 d x_2$$
			has the following properties
			\begin{itemize}
				\item[(1)] $\max_{\eta}\{t_1, t_2\}$ is non decreasing in all variables, smooth and convex on $\mr^2$;
				\item[(2)] $\max\{t_1, t_2\} \leq  \max_{\eta}\{t_1, t_2\} \leq \max\{t_1 + \eta_1, t_2 + \eta_2\} $;
				\item[(3)] If $u_1, u_2$ are plurisubharmonic functions, then $\max_{\eta}\{u_1, u_2\}$ is also plurisubharmonic.
			\end{itemize}
			
		\end{lem}
        \subsection{Curvature positivity of Hermitian holomorphic vector bundles}\

		\indent Let $U\subset \mc_t^n$ be a domain, and let $(E,h^E)$ be a Hermitian holomorphic trivial vector bundle over $U$ of finite rank $r$.
		The Chern connection is now given by a collection of differential operators $\{D_{t_j}^E\}_{1\leq j\leq n}$, which satisfies
		$$
        \partial_{t_j}h^E(u, v) = h^E(D_{t_j}^E u, v) + h^E(u, \bar{\partial}_{t_j} v),\ \forall u,v \in C^2(U,E),
        $$
		where $\partial_{t_j}:= \frac{\partial}{\partial t_j}$ and $\bar{\partial}_{t_j}:= \frac{\partial}{\partial \bar{t}_j}$
        for any $1\leq j\leq n.$ For any $1\leq j,k\leq n,$ let $\Theta^{(E,h^E)}_{jk}:=[D_{t_j}^E, \bar{\partial}_{t_k}]$ (Lie bracket), then the Chern curvature of
       $(E,h^E)$ is given by
		$$
         \Theta^{(E,h^E)} =\sum_{j,k=1}^n \Theta^{(E,h^E)}_{jk} dt_j \wedge d\bar{t}_k.
        $$
		\begin{defn}\label{Curvature positivity of Berndtsson}
			The curvature of $(E,h^E)$ is said to be positive (resp. strictly positive) in the sense of Nakano if for any nonzero $n$-tuple $(u_1, \cdots  , u_n)\\
           \in C^\infty(U,E),$
            we have
			$$
            \sum_{j,k=1}^n h^E\left(\Theta_{jk}^{(E,h^E)} u_j, u_k\right) \geq 0\ (\text{resp.} >0).
            $$
		\end{defn}
         We need another simple lemma in the proof of Theorem \ref{thm(extend):curvartue positive circular domains}.
		\begin{lem}[see {\cite[Theorem (14.5), Chapter V]{Dem}}]\label{the direct sum of holomorphic vector bundles}
			Let $(F,h)$ be a Hermitian holomorphic vector bundle over $X$, and let $E, G$ be two holomorphic subbundles of $F$ such that
			$F=E\oplus G$ and $E$ is orthogonal to $G$, then the curvature of these bundles satisfies
			$$\Theta^{(F,h)}=\Theta^{(E,h)} \oplus \Theta^{(G,h)}.$$
		\end{lem}
           Although the above lemma is originally stated for finite rank bundles, it can be easily generalized to infinite rank bundles.
        \subsection{$L^2$-estimates in complex analysis}\

        In this subsection, we collect some well known $L^2$-estimates in complex analysis that will be used later.
        \begin{lem}[{\cite[Lemma 1.6.4]{Ber95}}]\label{thm:Berndtsson L^2 estimate}
         Let $\Omega\subset \mc^n$ be a pseudoconvex domain, let $\varphi$ be a plurisubharmonic function on $\Omega,$ and let $\psi\in C^2(\Omega)$ be a  strictly plurisubharmonic function. Then for any $\bar{\partial}$-closed measurable $(0,1)$-form $f:=\sum_{j=1}^{n}f_jd\bar z_j$ on $\Omega$, we can solve $\bar\partial u=f$ with the estimate
        $$\int_{\Omega}|u|^2 e^{-\varphi-\psi}d\lambda \leq\int_{\Omega}\sum_{j,k=1}^{n}\psi^{j\bar k}f_j\bar{f_k}e^{-\varphi-\psi}d\lambda,$$
        provided the right hand side is finite.
        \end{lem}

        \begin{lem}[{\cite[Theorem 1]{Blo13}}]\label{thm:extension}
         Let $0\in \Delta\subset \mc$ be a bounded domain, let $g_\Delta(\cdot,z)$ be the (negative) Green function for $\Delta$
         with  pole at $z,$ and let $\Omega\subset \Delta\times \mc^{n}$ be a pseudoconvex domain. Then for any plurisubharmonic function $\varphi$ on $\Omega,$ holomorphic function $f$ on $H:=\{(z_1,\cdots,z_{n+1})\in \Omega|\ z_1=0\},$ there exists $F\in\mo(\Omega)$ such that
         $F|_{\Omega'}=f$ and
         $$\int_{\Omega}|F|^2e^{-\varphi}d\lambda\leq \frac{\pi}{(c_\Delta(0))^2}\int_{H}|f|^2e^{-\varphi}d\lambda,$$
         provided the right hand side is finite, where
         $$c_\Delta(0):=\exp(\lim_{z\rightarrow 0}(g_\Delta(z,0)-\log|z|)).$$
        \end{lem}
        A fundamental result about the $L^2$-estimate of $\bar\partial$ for a Hermitian holomorphic vector bundle is the following, which is due to H\"ormander and Demailly.
		\begin{lem}[{\cite[Theorem 6.1, Chapter VIII]{Dem}}]\label{lem: L2 estimate Nakano}
         Let $X$ be a complete K\"{a}hler manifold with a K\"{a}hler metric $\omega$ which is not necessarily complete. Let $(E,h^E)$ be a  holomorphic vector bundle over $X$ with a Nakano positive Hermitian metric $h^E$, and define the  operator $A:=i\Theta^{(E,h^E)}\wedge\Lambda_\omega$ acts on $\Lambda^{n,q}T^{*}X\otimes E$ ($\Lambda_\omega$ is the adjoint operator of the $E$-valued Lefschetz operator,
         $ q\geq 1$). Then for any $v\in L^2(X,\Lambda^{n,q}T^{*}X\otimes E)$ satisfying $\bar\partial v=0$, there exists $u\in  L^2(X, \Lambda^{n,q-1}T^*X\otimes E)$ such that $\bar\partial u=v$ and
        $$
        \int_X|u|^2_{h^E}d\lambda(\omega)\leq \int_X\langle A^{-1}v,v\rangle_{h^E}d\lambda(\omega),
        $$
        provided the right hand side is finite.
		\end{lem}
			The following result of Deng-Ning-Wang-Zhou shows that the converse of the above Lemma also holds,
		and hence gives an equivalent integral form characterization of the curvature positivity of Hermitian holomorphic vector bundles.
			\begin{lem}[{\cite[Theorem 1.1]{DNWZ}}]\label{lem: optimal L2-estimate}
			Let $U\subset \mc^n$ be a bounded domain with the standard K\"ahler metric $\omega:=i \sum_{j=1}^n dz_j \wedge d\bar{z}_j$, $(E, h)$ be a Hermitian holomorphic vector bundle over $U$
			with smooth Hermitian metric $h$,
			and let $\theta \in C^0 (U, \wedge^{1,1}T^*U\otimes \operatorname{End}(E))$ with $\theta^*=\theta$.
			If for any strictly plurisubharmonic function $\psi$ on $U$ and $v \in C_c^{\infty} (U, \wedge^{n,1}T^*U \otimes E )$ with $\bar{\partial}f=0$
			and $i\partial\bar\partial\psi\otimes \operatorname{Id}_E + \theta> 0$ on $\operatorname{supp}(v)$ (support of $v$),
			there is a measurable section $u$ of $\wedge^{n,0}T^*U \otimes E $, satisfying $\bar{\partial}u=v$ and
			\begin{equation}
				\int_U |u|_h^2e^{-\psi} d\lambda(z) \leq \int_U \langle B_{i\partial\bar\partial\psi, \theta}^{-1}v, v\rangle_{h}e^{-\psi} d\lambda(z),
			\end{equation}
			provided that the right hand side is finite, then $i\Theta^{E} \geq \theta$ in the sense of Nakano, where
			$$B_{i\partial\bar\partial\psi,\theta} := [i\partial\bar\partial\psi\otimes \operatorname{Id}_E + \theta, \Lambda_\omega].$$
		\end{lem}	
		
         Let us also note that the following simple lemma.
        \begin{lem}[{\cite[Theorem 5.2, Chapter VIII]{Dem}}]\label{thm:complete}
         Let $\Omega\subset \mc^n$ be a pseudoconvex domain, then $\Omega$ is a complete K\"ahler manifold.
        \end{lem}
		
		
		\subsection{Curvature positivity of real vector bundles}\
        \par In this subsection, we give the definition of curvature posivity of a real vector bundle, which is similar to Definition \ref{Curvature positivity of Berndtsson}. 
         \begin{defn}[{\cite[Definition 2]{Rau13}}]\label{defn:curvature operator}
            Let $g$ be a $C^2$ Riemannian metric on the trivial vector bundle $E:=D\times\mathbb{R}^r$ over an open subset $D\subset\mathbb{R}^n$, and let
            $$\Theta^{(E,g)}_{jk}:=-\frac{\partial}{\partial x^k}\left(g^{-1}\frac{\partial g}{\partial x^j}\right),\ 1\leq j,k\leq n,$$
            where differentiation should be interpreted elementwise. The curvature of $(E,g)$ is said to be positive (resp. strictly positive) in the sense of Nakano if for any nonzero $n$-tuple $(u_1, \cdots, u_n)\in C^\infty(U,\mr^n),$ we have
			$$\sum_{j,k=1}^n g(\Theta_{jk}^{(E,g)} u_j, u_k)\geq 0\ (\text{resp.} >0).$$
         \end{defn}
         \subsection{Some basic knowledge from measure theory}\ \\
          \indent In this subsection, we give some basic knowledge from measure theory that will be used several times
             when we consider applications to convex analysis.
            \begin{defn}
            Let $U\subset\mr^m,\ V\subset\mr^n$ be two measurable subsets, and let $\phi\colon U\rw V$ be a measurable map.
            For any $A\subset V$ which is  measurable, define $\mu(A):=|\phi^{-1}(A)|$, then $\mu$ is a measure on $V$,
            which is called the pushforward measure of the Lebesgue measure on $U$ via $\phi$.
            \end{defn}
            \begin{lem}\label{measure:1}
            Let $U\subset\mr^m,\ V\subset\mr^n$ be two  measurable subsets, let $\phi\colon U\rw V$ be a (Lebesgue) measurable map,
            and let $\mu$ be the pushforward measure of the Lebesgue measure on $U$ via $\phi$, then for any measurable function $f\colon V\rw [-\infty,\infty]$, we have
            $$\int_{U}fd\mu=\int_{\phi^{-1}(U)}f\circ\phi d\lambda$$
            provided one side of the above exists.
            \end{lem}
            \begin{lem}\label{measure:2}
            Let $U\subset\mr^n$ be an open subset, $\phi\colon U\rw \mr^n$ be a $C^1$ map, then for any integrable function $f$ on $\phi(U)$, we know
            $$\int_{\phi(U)}f(y)d\lambda(y)=\int_{U}f(\phi(x))J_\phi(x)d\lambda(x),$$
            where $J_\phi:=|\det(D(\phi))|$ is the absolute value of the Jacobian of $\phi$.
            \end{lem}
            Combining the above two lemmas, we have the following
            \begin{cor}\label{measure:3}
             Let $U\subset\mr^n$ be an open subset, $\phi\colon U\rw \mr^n$ be a $C^1$ map which is injective and has non-vanishing Jacobian
            on $U$, and let $\mu$ be the pushforward measure of the Lebesgue measure on $U$ via $\phi$, then the Radon-Nikodym derivative of $\mu$ with respect to
            the Lebesgue measure $\lambda$ on $V$ is
            $$\frac{d\mu}{d\lambda}=J_{\phi}^{-1}\circ\phi^{-1}.$$
            \end{cor}

			\section{The proof of Theorem \ref{thm:Ber}}\
             
The proof of Theorem \ref{thm:Ber} is as follows.
             \begin{proof}
            Let $\Omega:=V\times W.$ By the usual approximation technique, we may assume $\psi$ is of class $C^2$ on $\Omega$. Consider the map
             $$\operatorname{id}\times\phi\colon \mathbb{C}_\tau^m\times \mathbb{C}_z^n\rightarrow\mathbb{C}_\tau^m\times(\mathbb{C}_w^*)^n,\
            (\tau,z_1,\cdots,z_n)\rightarrow(\tau,e^{z_1},\cdots,e^{z_n}),$$
            and let $\Omega':=(\operatorname{id}\times \phi)(\Omega),$ then $\Omega'$ is a pseudoconvex domain. It is clear that $\psi$ induces a $C^2$, plurisubharmonic function $\psi'$ on $\Omega'$ by
            $$\psi'(\tau,e^{z_1},\cdots,e^{z_n})=\psi(\tau,z_1,\cdots,z_n),$$
            and $\psi'$ is  strictly plurisubharmonic with respect to $\tau.$  Similarly, $\varphi$ also induces a plurisubharmonic function $\varphi'$ on $\Omega'$. We may regard $f$ as a $\bar{\partial}$-closed $(0,1)$-form on $\Omega'$. For any measurable subset $U\subset (\mc_w^*)^n,$ define
            $$
            \mu(U):=\left|\{z\in \mc^n|\ \phi(z)\in U,\ \im(z_1),\im(z_2),\cdots,\im(z_n)\in [0,2\pi)\}\right|,
            $$
            then $\mu$ is a measure on $(\mc_w^*)^n,$ and by Corollary \ref{measure:3}, we know
            $$
            d\mu(w)=\frac{1}{|w_1|^2\cdots|w_n|^2}d\lambda(w).
            $$
            \indent  By Lemma \ref{measure:1}, we have
            \begin{equation}\label{equ:1}
            \int_{\Omega'}\sum_{j,k=1}^{n}(\psi')^{j\bar k}f_j\bar{f}_ke^{-\varphi'-\psi'}d\mu
            =(2\pi)^n\int_{V\times D}\sum_{j,k=1}^{n}\psi^{j\bar{k}}f_j\bar{f}_ke^{-\varphi-\psi}d\lambda<\infty.
            \end{equation}
            Let $\varphi''(w):=\varphi'(w)+\sum_{j=1}^n\log|w_j|^2$ for any $w\in\Omega'$, then $\varphi''$ is a plurisubharmonic function on $\Omega'$,
            and
            $$
            e^{-\varphi'}d\mu(w)=e^{-\varphi''}d\lambda(w),\ \forall w\in (\mc^*)^n.
            $$
            By Lemma \ref{thm:Berndtsson L^2 estimate}, we can solve $\bar\partial_{(\tau,w)} u=f$ on $\Omega'$  with the estimate
            $$
            \int_{\Omega'}|u|^2 e^{-\varphi''-\psi'}d\lambda \leq\int_{\Omega'}\sum_{j,k=1}^{n}(\psi')^{j\bar k}f_j\bar{f}_ke^{-\varphi''-\psi''}d\lambda,
            $$
            i.e.
            \begin{equation}\label{equ:2}
            \int_{\Omega'}|u|^2 e^{-\varphi'-\psi'}d\mu\leq\int_{\Omega'}\sum_{j,k=1}^{n}(\psi')^{j\bar k}f_j\bar{f}_ke^{-\varphi'-\psi'}d\mu,
            \end{equation}
            Take $u$ to be minimal, then $u$ is independent of the argument of $w$ by the uniqueness of minimal solution. As $f$ is independent of $w$, we know $u$ is holomorphic in $w$, so $u$ is independent of $w$. Thus, we can regard $u$ as a function on $V$, then we have $\bar{\partial}u=f$. By Lemma \ref{measure:1} again, we have
            \begin{equation}\label{equ:3}
            \int_{\Omega'}|u|^2e^{-\varphi'-\psi'}d\mu=(2\pi)^n\int_{V\times D}|u|^2e^{-\varphi-\psi}d\lambda.
            \end{equation}
            Combining Formulas (\ref{equ:1}), (\ref{equ:2}), (\ref{equ:3}), we complete the proof.
            \end{proof}
            \section{The proof of Theorem \ref{thm:optimal}}
             In this section, we give and prove a more general version of Theorem \ref{thm:optimal}.
            \begin{thm}\label{thm:general optimal2}
             Let $D\subset \mr_x^n$ be a convex domain, $0\in \Delta\subset \mc_\tau$ be a bounded domain, and let $\varphi(\tau,z)$ be a plurisubharmonic function on $\Delta_\tau\times (D_x+i \mr^n_y),$ which is independent of $\im(z)$. Then there exists a holomorphic function $f$ on $\Delta$ satisfying $f(0)=1$ and
            $$
            \int_{\Delta\times D}|f(\tau)|^2e^{-\varphi(\tau,x)}d\lambda(\tau,x)\leq \frac{\pi}{(c_\Delta(0))^2}\int_D e^{-\varphi(0,x)}d\lambda(x),
            $$
            provided the right hand side is finite.
            \end{thm}
            \begin{proof}
            Let $W_z:=D_x+i \mr^n_y\subset \mc_z^n,\ \Omega:=\Delta\times W.$  Consider the map
             $$\operatorname{id}\times\phi\colon \mathbb{C}_\tau^m\times \mathbb{C}_z^n\rightarrow\mathbb{C}_\tau^m\times(\mathbb{C}_w^*)^n,\
            (\tau,z_1,\cdots,z_n)\rightarrow(\tau,e^{z_1},\cdots,e^{z_n}),$$
            and let $\Omega':=(\operatorname{id}\times \phi)(\Omega),$ then $\Omega'$ is a pseudoconvex domain.  It is clear that $\varphi$ induces a $C^2$, plurisubharmonic function $\varphi'$ on $\Omega'$ by
            $$\varphi'(\tau,e^{z_1},\cdots,e^{z_n})=\varphi(\tau,z_1,\cdots,z_n).$$
            Similar as in the proof of Theorem \ref{thm:Ber}, we may get a measure on $(\mc_w^*)^n$ via $\phi$ which satisfies
            $$
            d\mu(w)=\frac{1}{|w_1|^2\cdots|w_n|^2}d\lambda(w),\ \forall w\in (\mc^*)^n.
            $$
            Let $\varphi''(w):=\varphi'(w)+\sum_{j=1}^n\log|w_j|^2$ for any $w\in\Omega'$, then $\varphi''$ is a plurisubharmonic function on $\Omega'$,
            and
            $$
            e^{-\varphi'}d\mu(w)=e^{-\varphi''}d\lambda(w).
            $$
            \indent Let $H:=\Omega'\cap \{\tau=0\}=\{0\}\times \pi(W),$  then by Lemma \ref{measure:1}, we have
            \begin{equation}\label{equ:4}
            \int_{H}e^{-\varphi''(0,w)}d\lambda(w)=\int_{H}e^{-\varphi'(0,w)}d\mu(w)=(2\pi)^n\int_D e^{-\varphi(0,x)}d\lambda(x)<\infty
            \end{equation}
            By Lemma \ref{thm:extension}, we know there is a holomorphic function $f$ on $\Omega'$ such that $f|_H=1$ and
            \begin{equation}\label{equ:5}
            \int_{\Omega'}|f(\tau,w)|^2e^{-\varphi''(\tau,w)}d\lambda(\tau)d\lambda(w)\leq \frac{\pi}{(c_\Delta(0))^2}\int_{H}e^{-\varphi''(0,w)}d\lambda(w),
            \end{equation}
            i.e.
            \begin{equation}\label{equ:5}
            \int_{\Omega'}|f(\tau,w)|^2e^{-\varphi'(\tau,w)}d\lambda(\tau)d\mu(w)\leq \frac{\pi}{(c_\Delta(0))^2}\int_{H}e^{-\varphi'(0,w)}d\mu(w).
            \end{equation}
            Take $f$ to be minimal, then we know $f$ is independent of the argument of $w$,
            and then $f$ is independent of $w$ as $f$ is holomorphic in $w.$ We may thus regard $f$ as a holomorphic function on $\Delta,$ then
            $f(0)=1.$ By Lemma \ref{measure:1} again, we have
            \begin{equation}\label{equ:6}
            \int_{\Omega'}|f(\tau,w)|^2e^{-\varphi'(\tau,w)}d\lambda(\tau)d\mu(w)=(2\pi)^n\int_{\Delta\times D}|f(\tau)|^2e^{-\varphi(\tau,x)}d\lambda(\tau,x).
            \end{equation}
            Therefore, the proof is complete by combining Formulas (\ref{equ:4}), (\ref{equ:5}), (\ref{equ:6}).
            \end{proof}
            \section{The proof of Theorem \ref{thm:Cor}}
            To prove Theorem \ref{thm:Cor}, let us firstly recall the complexification of a real inner product.
             Let $(V,g)$ be a finite dimensional real inner product space, and let $V_\mc:=V\otimes_\mr\mc$ be the complexification of
             $V$, then we can extend $g$ to a Hermitian inner product $h$ on $V_\mc$ by setting
             $$h(x\otimes z, y\otimes w)=z\bar{w}g(x,y),\ \forall x,y\in V,\ \forall z,w\in\mc.$$
             When we define a Hermitian inner product $h$ as above, we always say $g$ extends to a Hermitian inner product $h$ for brevity.\\
             \indent Now we give the proof of Theorem \ref{thm:Cor}.
            \begin{proof}
            Let $\Omega_z:=D_x+i \mr_y^n\subset \mc_z^n.$
            Consider the map
            $$\phi\colon \mc_z^n\rw (\mc_w^*)^n,\ (z_1,\cdots,z_n)\mapsto (e^{z_1},\cdots,e^{z_n}),$$
            and let $\Omega':=\phi(\Omega),$ then $\Omega'$ is a complete K\"ahler manifold with the canonical K\"ahler metric $\omega$ by Lemma \ref{thm:complete}.
            Similar as in the proof of Theorem \ref{thm:Ber}, we may get a measure on $(\mc_w^*)^n$ via $\phi$ which satisfies
            $$
            d\mu(w)=\frac{1}{|w_1|^2\cdots|w_n|^2}d\lambda(w),\ \forall w\in (\mc^*)^n.
            $$
            Let $E':=\Omega'\times \mc^r$ be the trivial vector bundle over $\Omega'.$
            We extend $g$ to a Hermitian metric $h$ on the trivial vector bundle $\Omega\times\mc^r$ over $\Omega$ such that $h$ is independent of $\im (z)$ for any $z\in\Omega$, then $h$ induces a Hermitian metric $h'$ on $E'$ via $\phi.$ Similar to the proof of Theorem \ref{thm:Ber}, we know $f$ induces a $\bar{\partial}$-closed $(0,1)$-form
            $f'\in L^2(\Omega',\Lambda^{0,1}T^*\Omega'\otimes E').$  For any $w\in\Omega'$, we set
            $$
            (h'')_w:=\frac{1}{|w_1|^2\cdots|w_n|^2}(h')_w.
            $$
            \indent  Since $(E,g)$ is Nakano positive, then
            $$
            A_{E'}:=[i\Theta^{(E',h'')},\Lambda_\omega]\geq 0\text{ on }\Lambda^{0,1}T^*\Omega'\otimes E'
            $$
            by a simple computation. By Lemma \ref{measure:1}, we have
            \begin{align}\label{equ:7}
            \int_{\Omega'}\langle A^{-1}_{E'}f',f'\rangle_{h''}d\lambda&=\int_{\Omega'}\langle A^{-1}_{E'}f',f'\rangle_{h'}d\mu=(2\pi)^n\int_{D}\langle(\Theta^{(E,g)})^{-1}f,f\rangle_{g}d\lambda\\
            &<\infty.\nonumber
            \end{align}
            By Lemma \ref{lem: L2 estimate Nakano}, we may solve $\bar{\partial}u'=f'$ with the estimate
            \begin{equation}\label{equ:8}
            \int_{\Omega'}|u'|^2_{h''}d\lambda\leq \int_{\Omega'}\langle A_{E'}^{-1}f',f'\rangle_{h''}d\lambda.
            \end{equation}
            Take $u'$ to be minimal, then $u'$ is independent of $\im(w)$. For any $x\in D$, define $u(x):=u'(\pi(x)),$
            then $du=f$. By Lemma \ref{measure:1} again, we know
            \begin{equation}\label{equ:9}
            \int_{\Omega'}|u'|^2_{h'}d\mu=(2\pi)^n\int_D|u|_g^2d\lambda.
            \end{equation}
            Thus, the desired  result follows from Formulas (\ref{equ:7}), (\ref{equ:8}), (\ref{equ:9}).
            \end{proof}
            \section{The proof of Theorem \ref{thm(extend):curvartue positive circular domains}}
             In this section, we will give the proof of Theorem \ref{thm(extend):curvartue positive circular domains}. 
              We first recall the following involved notion.
              
              \begin{defn}\label{def:strict p.s.c family}
			Let $U\subset\mc^n$ be a domain, and let $p\colon\mc^n\times\mc^m\ra\mc^n$ be the natural projection.\bi
			\item[(1)] A \emph{family of bounded domains} (of dimension $m$ over $U$) is a domain $\Omega\subset U\times\mc^m$ such that $p(\Omega)=U$
             and all fibers $\Omega_t:=p^{-1}(t)\subset\mc^m\ (t\in U)$ are bounded.  
			\item[(2)] A family of bounded domains $\Omega$ over $U$ has\emph{ $C^2$  boundary }if  there exists a $C^2$-function $\rho$
			defined on $U\times\mc^m$ such that $\Omega=\{(t,z)\in U\times\mc^m|\ \rho(t,z)<0\}$ and $d(\rho|_{\Omega_t})\neq 0$
            on $\partial\Omega_t$ for all $t\in U$. Such a function $\rho$ is  called a \emph{defining function} of $\Omega$.
			\item[(3)] A family of bounded domains $\Omega\subset U\times \mc^m$ is called \emph{strictly pseudoconvex} if it admits a $C^2$-boundary defining function that is
            strictly plurisubharmonic on some neighborhood of $\overline{\Omega}$ in $U\times\mc^m$.
			\ei
		\end{defn}

            \indent In this section, we always let $j,k=1,\cdots, n$ represent the indices of the components of $t=(t_1,\cdots, t_n)\in U$,
            $p,q,s,a=1, \cdots, m$ represent the indices of the components of  $z=(z_1,\cdots, z_m)\in \mc^m$, $\lambda,\mu,\alpha,\beta,\gamma=1,\cdots, r$ represent the indices of the components of $F$. To simplify the notation, we always write $dz$ for $dz_1\wedge\cdots\wedge dz_m$.\\
            \indent Write
            $$D_{t_j}^Fe_\lambda=\sum_{\mu}\Gamma_{j\lambda}^\mu dt_j\otimes e_\mu,\ D_{z_p}^Fe_\lambda=\sum_{\mu}\Gamma_{p\lambda}^\mu dz_p\otimes e_\mu,$$
            and we also adopt the following notations
            $$
            \left(\begin{array}{cc}
            	(A_{pq\lambda}^\alpha)& (A_{pj\lambda}^\alpha)\\
            	(A_{jp\lambda}^\alpha) & (A_{jk\lambda}^\alpha)
            \end{array}\right):=\left(\begin{array}{cc}
            	\left(-\frac{\partial\Gamma_{p\lambda}^\alpha}{\partial\bar{z}_q}\right)& \left(-\frac{\partial\Gamma_{p\lambda}^\alpha}{\partial\bar{t}_j}\right)\\
            	\left(-\frac{\partial\Gamma_{j\lambda}}{\partial\bar{z}_p}\right)& \left(-\frac{\partial\Gamma_{j\lambda}^\alpha}{\partial\bar{t}_k}\right)
            \end{array}\right),\ \ h^F_{\lambda\mu}:=h_{(t,z)}^F(e_\lambda,e_\mu),$$
            $$
            \left(\begin{array}{cc}
            	(A_{pq\lambda\mu})& (A_{pj\lambda\mu})\\
            	(A_{jp\mu\lambda}) & (A_{jk\lambda\mu})
            \end{array}\right):=\left(\begin{array}{cc}
            	\left(-\frac{\partial\Gamma_{p\lambda}^\alpha}{\partial\bar{z}_q}h^F_{\alpha\mu}\right)& \left(-\frac{\partial\Gamma_{p\lambda}^\alpha}{\partial\bar{t}_j}h^F_{\alpha\mu}\right)\\
            	\left(-\frac{\partial\Gamma_{p\lambda}^\alpha}{\partial\bar{t}_j}h^F_{\alpha\mu}\right)^*& \left(-\frac{\partial\Gamma_{j\lambda}^\alpha}{\partial\bar{t}_k}h^F_{\alpha\mu}\right)
            \end{array}\right),$$
            where $^{*}$ represents the conjugate of the transpose of a matrix. It is clear that
            $$\Theta_{jk}^F=[D_{t_j}^F,\bar{\partial}_{t_k}]=\sum_{\lambda,\alpha}A_{jk\lambda}^\alpha e_\lambda^*\otimes e_\alpha,$$
            where $e_1^*,\cdots,e_r^*$ is the dual basis of $e_1,\cdots,e_r$. \\
            \indent To prove Theorem \ref{thm(extend):curvartue positive circular domains}, we need the following key lemma,
            which is a generalization of \cite[Formula (3.1)]{Ber09}.
             \begin{lem}\label{lem:key lemma}
            	Let $U\subset\mc^n$ be a pseudoconvex domain, and let $D\subset\mc^m$ be a pseudoconvex bounded domain.
            	let $(F,h^F)$ be a Nakano positive  trivial vector bundle of finite rank $r$ defined on some neighborhood of the closure of $\Omega:=U\times D$. Let $\{e_1, \cdots, e_r\}$ be the canonical holomorphic frame of $(F,h^F)$. For any $t\in U$, set
            	$$E_t:=\{f=\sum_{\lambda=1}^r f_\lambda \otimes e_\lambda |\ f_\lambda \in \mathcal O(D)\text{ for all }\lambda,\ \|f\|^2_t<\infty\},$$
            	where
            	$$\|f\|^2_t:=\int_{D}\sum_{\lambda,\mu=1}^rf_\lambda(z)\overline{f_\mu(z)}h^F_{(t,z)}(e_\lambda,e_\mu)d\lambda(z),$$
            	and set $E:=\cup_{t\in U}E_t$, then $E$ is a holomorphic vector bundle (of infinite rank) over $U$ with a Hermitian metric $h$ given by
            	$$h_t^E(f,g):=\int_{D}\sum_{\lambda,\mu=1}^rf_\lambda(z)\overline{g_\mu(z)}h^F_{(t,z)}(e_\lambda,e_\mu)d\lambda(z),\ \forall f,g\in E_t.$$
            	Then for any smooth sections $u_j:=\sum_{\lambda}u_{j\lambda}dz\otimes e_\lambda(1\leq j\leq n)$ of $E$, we have
            	\begin{align*}
            		\sum_{j,k}h^E\left(\Theta^{E}_{jk}u_j,u_k\right)
            		&\geq \int_D \sum_{j,k,\beta,\gamma}\left(A_{jk\beta\gamma}-\sum_{p,q,\lambda,\mu}B_{pq\lambda\mu}A_{jq\beta\mu}A_{pk\lambda\gamma}\right)u_{j\beta}\overline{u_{k\gamma}} d\lambda\\
            		&=:\int_D \sum_{j,k,\beta,\gamma} H(h^F)_{jk\beta \gamma} u_{j\beta}\overline{u_{k\gamma}} d\lambda,
            	\end{align*}
            	where $B$ is uniquely determined by $A$ and satisfies
            	$$\sum_{q,\beta}A_{pq\lambda\beta}B_{sq\alpha\beta}=\delta_{ps}\delta_{\lambda\alpha}.$$
            \end{lem}
            \begin{proof}
            For any $t\in U$, set
            $$G_t:=\{u:=\sum_{\lambda} u_\lambda dz\otimes e_\lambda|\ u_\lambda \in L^2(D)\text{ for all }\lambda\}$$
    		with an inner product $h^G_t$ given by
    		$$h_t^G(u,v):=\int_{\Omega_t} \sum_{\lambda,\mu}u_\lambda(z)\overline{v_\mu(z)}h^F_{(t,z)}(e_\lambda,e_\mu)d\lambda(z),\ \forall u,v \in G_t,$$
    		then $G:=\cup_{t\in U}G_t$  can be viewed as a Hermitian holomorphic (trivial) vector bundle over $U.$ Let $\pi\colon G\rw E$ be the fiberwise orthogonal projection, and
            let $\pi_\perp$ be the orthogonal projection on the orthogonal complement of $E$ in $G.$ Put
            $$w:=\pi_\perp(\sum_j D^F_{t_j}u_j),\ \|w\|_{h^F}^2:=\int_{D}h^F(w,w)d\lambda,$$
            then we have
                 $$D^{E}=\pi \circ D^G$$
		    	and
		    	\begin{align*}
		    		&\quad \sum_{j,k}h^E\left(\Theta^{E}_{jk}u_j,u_k\right)\\
                   &=\sum_{j,k} h^G\left(\Theta^G_{jk}u_j,u_k\right)-h^G\left(\pi_\perp\sum_j D^G_{t_j}u_j,\pi_\perp\sum_j D^G_{t_j}u_j\right)\\
		    		&=\sum_{j,k}\int_{D}h^F\left(\Theta^F_{jk}u_j,u_k\right)d\lambda-\left\|w\right\|_{h^F}^2
		    	\end{align*}
		    	for any smooth sections $u_1, \cdots, u_n$ of $E$. For fixed $t\in U$, $w$ solves the $\dbar_z$-equation
		    	$$\dbar_z w =\dbar_z \left(\pi_\perp\sum_j D^F_{t_j}u_j \right)=\dbar_z \left(\sum_j D^F_{t_j}u_j \right).$$
		    	\indent We will use Lemma \ref{lem: L2 estimate Nakano} to give an estimate of $\|w\|_{h^F}^2$.
                By Lemma \ref{thm:complete}, we know $D$ is a complete K\"ahler manifold. Let $\omega$ be the standard K\"ahler metric on $D$,
                and set $A:=[i\Theta^{(F_t,h^F)},\Lambda_\omega]$, then we know $A\geq 0$ as $(F,h^F)$ is Nakano strictly positive.
                 Set $v:=(-1)^n\dbar_z w \in L^2(D,\wedge^{n,1}T^*D\otimes F_t),$ then $v$ satisfies $\dbar_z v=0.$ By Lemma \ref{lem: L2 estimate Nakano}, we have
		    	$$\|w\|_{h^F}^2 \leq \int_{D} \langle A^{-1}v,v\rangle_{h^F} d\lambda(\omega),$$
               then
		    	$$\sum_{j,k}h^E\left(\Theta^{E}_{jk}u_j,u_k\right)\geq \sum_{j,k}\int_Dh^F\left(\Theta^F_{jk}u_j,u_k\right)d\lambda-\int_{D} \langle A^{-1}v,v \rangle_{h^F} d\lambda.$$
                Now we need to compute the right hand side the above inequality. To do so, we need some simple but long computation.\\
		    	\indent It is clear that
		    	$$\sum_{j,k}h^F\left(\Theta_{jk}^Fu_j,u_k\right)=\sum_{j,k,\lambda,\mu}A_{jk\lambda}^\alpha u_{j\lambda}\overline{u_{k\mu}}h^F_{\alpha\mu}
		    	=\sum_{j,k,\lambda,\mu}A_{jk\lambda\mu}u_{j\lambda}\overline{u_{k\mu}},$$
		    	and
		    	\begin{align*}
		    		v&=\bar{\partial}_z((-1)^n\sum_jD_{t_j}^Fu_j)=\sum_{j,\lambda,\mu}u_{j\lambda}\frac{\partial\Gamma_{j\lambda}^\mu}{\partial\bar{z}_p}dz\wedge d\bar{z}_p\otimes e_\mu\\
		    		&=\sum_{j,\lambda,\mu}u_{j\mu}A_{jp\mu}^\lambda dz\wedge d\bar{z}_p\otimes e_\lambda.
		    	\end{align*}
                 Hence it suffices to prove that
                $$\langle A^{-1}v,v\rangle_{h^F}=\sum_{j,k,p,q,\lambda,\mu,\beta,\gamma}B_{pq\lambda\mu}A_{jq\beta\mu}A_{pk\lambda\gamma}u_{j\beta}\overline{u_{k\gamma}}$$
                for some $B$ to be determined later.\\
		    	\indent For any
		    	$$g:=\sum_{p,\lambda}g_{p\lambda}dz\wedge d\bar{z}_p\otimes e_\lambda \in L^2(D,\wedge^{n,1}T^*D\otimes F_t),$$
		    	we have
		    	\begin{align*}
		    		Ag&=i\Theta^{(F_t,h^F)}(\Lambda_\omega g)=\left(i\sum_{p,q,\lambda,\mu}A_{pq\lambda}^\mu dz_p\wedge d\bar{z}_q\otimes e_\lambda^*\otimes e_\mu\right)(\Lambda_\omega g)\\
		    		&=\sum_{p,q,j,\lambda,\mu}A_{pq\lambda}^\mu g_{p\lambda}dz\wedge d\bar{z}_q\otimes e_\mu,
		    	\end{align*}
                 where $e_1^*,\cdots,e_r^*$ is the dual frame of $e_1,\cdots,e_r.$ Note that we may view $A$ as a linear map $T$ which satisfies
		    	$$T(d\bar{z}_p\otimes e_\lambda)= \sum_{q,\mu}A_{pq\lambda}^\mu d\bar{z}_q\otimes e_\mu.$$
		    	Since $(F,h^F)$ is Nakano strictly positive, then  for any nonzero matrices $u_{p\lambda}$, we have
		    	$$\sum_{p,q,\lambda,\mu}A_{pq\lambda\mu}u_{p\lambda}\overline{u_{q\mu}}>0,$$
		    	and then $T$ is invertible. Let $T^{-1}$ be the inverse map of $T$, and write
		    	$$T^{-1}(d\bar{z}_q\otimes e_\mu)=\sum_{p,\lambda}B_{pq\lambda}^\mu d\bar{z}_p\otimes e_\lambda,$$
		    	then we have
		    	$$d\bar{z}_p\otimes e_\lambda=T^{-1}(\sum_{q,\mu}A_{pq\lambda}^\mu d\bar{z}_q\otimes e_\mu)=\sum_{q,s,\alpha,\mu}A_{pq\lambda}^\mu B_{sq\alpha}^\mu d\bar{z}_s\otimes e_\alpha,$$
		    	i.e.,
		    	$$\sum_{q,\mu}A_{pq\lambda}^\mu B_{sq\alpha}^\mu=\delta_{ps}\delta_{\lambda\alpha}.$$
		    	Let
		    	$$B_{sq\alpha\beta}:=\sum_{\mu} B_{sq\alpha}^\mu (h^F)^{\beta\mu},$$
                where $\left((h^F)^{\beta\mu}\right)_{1\leq \beta,\mu\leq r}$ is the inverse matrix of $\left(h_{\beta\mu}^F\right)_{1\leq \beta,\mu\leq r}$, then we have
		    	$$\sum_{q,\beta}A_{pq\lambda\beta}B_{sq\alpha\beta}=\sum_{q,\mu}A_{pq\lambda}^\mu B_{sq\alpha}^\mu=\delta_{ps}\delta_{\lambda\alpha}.$$
		    	Define linear maps
		    	$$S_1(d\bar{z}_p\otimes e_\lambda)=\sum_{q,\mu}A_{pq\lambda\mu}d\bar{z}_q\otimes e_\mu,\ S_2(d\bar{t}_j\otimes e_\lambda)=\sum_{q,\mu}A_{jq\lambda\mu}d\bar{z}_q\otimes e_\mu,$$
		    	$$S_3(d\bar{z}_p\otimes e_\lambda)=\sum_{k,\mu}A_{pk\lambda\mu}d\bar{t}_k\otimes e_\mu,\ S_4(d\bar{t}_j\otimes e_\lambda)=\sum_{k,\mu}A_{jk\lambda\mu}d\bar{t}_k\otimes e_\mu,$$
		    	$$B(d\bar{z}_q\otimes e_\mu)=\sum_{p,\lambda}B_{pq\lambda\mu}d\bar{z}_p\otimes e_\lambda,$$
		    	then we have
		    	$$B S_1(d\bar{z}_p\otimes e_\lambda)=B\left(\sum_{q,\mu}A_{pq\lambda\mu}d\bar{z}_q\otimes e_\mu\right)=
		    	\sum_{q,s,\mu,\alpha}A_{pq\lambda\mu}B_{sq\alpha\mu}d\bar{z}_s\otimes e_\alpha=d\bar{z}_p\otimes e_\lambda,$$
		    	so $B=S_1^{-1}$. \\
		    	\indent Let
                 $$v_{q\lambda}:=\sum_{j,\mu}u_{j\mu}A_{jp\mu}^\lambda,$$
                 then we know
		    	$$A^{-1}v=\sum_{p,q,\lambda,\mu} B_{pq\lambda}^\mu v_{q\mu}dz_1\wedge\cdots\wedge dz_n\wedge d\bar{z}_p\otimes e_\lambda,$$
		    	and
		    	\begin{align*}
                \langle A^{-1}v,v\rangle_{h^F}&=\sum_{p,q,\lambda,\mu,\alpha}B_{pq\lambda}^\mu v_{q\mu}\overline{v_{p\alpha}}h^F_{\lambda\alpha}=\sum_{p,q,j,k,\lambda,\mu,\beta,\gamma}B_{pq\lambda}^\mu A_{jq\beta}^\mu \overline{A_{kp\gamma\lambda}}u_{j\beta}\overline{u_{k\gamma}}\\
                &=\sum_{j,k,p,q,\lambda,\mu,\beta,\gamma}B_{pq\lambda\mu}A_{jq\beta\mu}A_{pk\lambda\gamma}u_{j\beta}\overline{u_{k\gamma}}.
                \end{align*}
                Then the proof is complete.
            \end{proof}

 Recall that a domain $D\subset \mc^m$ is called a circular domain if it is invariant under the action of $\mathbb{S}^1$
 on $\mc^m$ given by
        $$e^{i\theta}\cdot(z_1,\cdots, z_m):=(e^{i\theta}z_1,\cdots, e^{i\theta}z_m),\ \theta\in\mr,$$
        and is called a Reinhardt domain if it is invariant under the action of the torus group $\mathbb{T}^m$
        on $\mc^m$ given by
        $$(e^{i\theta_1},\cdots, e^{i\theta_m})\cdot(z_1,\cdots, z_m):=(e^{i\theta_1}z_1,\cdots, e^{i\theta_m}z_m),\ \theta_i\in\mr.$$
        
            We now give the proof of  Theorem \ref{thm(extend):curvartue positive circular domains}.
            \begin{proof}
            	Since $\Omega$ is strictly pseudoconvex with $C^2$-boundary, there is a defining function $\rho$
            	which is $C^2$-smooth and strictly plurisubharmonic on some neighborhood of $\overline{\Omega}$ and $\mathbb{S}^1$-invariant with respect to $z$.
            	For any fixed $t_0\in U$ and any $q>0$, let $D:=\{z\in \mc^n|\ \rho(t_0,z)<q\}$, then there exists a  neighborhood $U'\subset\subset U$ of $t_0$
            	and $0<q<1$ such that
            	\begin{itemize}
            		\item[(1)] $\rho$ is defined, $C^2$-smooth and strictly plurisubharmonic on some neighborhood of the closure of $U'\times D$;
            		\item[(2)] $(F,h^F)$ is defined and Nakano positive on some neighborhood of the closure of $U'\times D$;
            		\item[(3)] let $p\colon \mc^n\times\mc^m\ra\mc^n$ be the natural projection, then $p^{-1}(U')\cap\Omega\subset U'\times D$;
            		\item[(4)] $D$ is a pseudoconvex circular domain;
            		\item[(5)] $\Omega_{t_0}\subset\subset D.$
            	\end{itemize}
            	Since the result to be proved is local, we may assume
            	$$U=U',\ \Omega\subset U\times D=:\tilde\Omega,\ \Omega_{t_0}\subset\subset D.$$
            \indent For any $N \in \mathbb{N},\ \epsilon>0$, we set
			$$h^F_N:= h^Fe^{-N\max\nolimits_{\left(\frac{1}{N^3}, \frac{1}{N^3}\right)} \{0, \rho\}},\ h^F_{N, \epsilon}:=h_{N}^Fe^{-\epsilon|t|^2-\epsilon|z|^2},$$
           where $\max\nolimits_{(-,-)}$ is the regularized max function defined as in Lemma \ref{lem: Regularization max functions}.
           It is clear that $(F,h_{N,\epsilon}^F)$ is Nakano strictly positive on some neighborhood of the closure of $\tilde{\Omega}$.\\
            	\indent  Let $E$ be the trivial vector bundle over $U$ as in Lemma \ref{lem:key lemma}.
            For any $t\in U,\ f,g\in E_t$, set $$\left(h_N^E\right)_t(f,g):=\int_{\Omega_t}\sum_{\lambda,\mu=1}^rf_\lambda\overline{g_\mu} \left(h_N^F\right)_{(t,z)}(e_\lambda,e_\mu)d\lambda(z),$$
            then $h_N^E$ is a Hermitian metric on $E$. We may similarly define $h_{N,\epsilon}^E$ according to $h_{N,\epsilon}^F.$
            It is clear that  $E^k$ is a holomorphic subbundle of $E$ (the metric on $E^k$ is similar to that on $E$). Similar to the proof of \cite[Lemma 3.1]{DHQ}, by Lemma \ref{lem:key lemma}, we
            know for any sections $u_1, \cdots  , u_n$ of $E^k$, we have
			\begin{equation}\label{curvature inequality}
                \sum_{j,l}\left. h_{N,\epsilon}^{E}\left(\Theta_{jl}^{(E^k, h_{N,\epsilon}^E)} u_j, u_l\right)\right|_{t_0} \geq
                             \int_D \sum_{j,l,\beta,\gamma} H(h^F_{N,\epsilon})_{jl\beta\gamma}(t_0,z)u_{j\beta} \overline{u_{l\gamma}}d\lambda(z).
             \end{equation}
            In fact, since $D$ is a circular domain containing the origin, any $f\in \mathcal O(D)$ can be represented as  a series
            	$$f=\sum_{j=0}^{+\infty}f_j$$
            that is convergent locally uniformly on $D$, where each $f_j$ is a homogenous polynomial of degree $j$. For any $\mathbb{S}^1$-invariant continuous bounded function $\psi$ on $D$, and any homogenous polynomials $g_j, g_l$ of degree $j$ and $l$ respectively,
            	we have
            	$$\int_Dg_j\bar g_l \psi d\lambda=0,\ \forall j\neq l.$$
            	It follows that, for any $t\in U$, an element $f$ in the orthogonal complement $(E^k)_t^{\bot}$ of $E_t^k$ in $E_t$ has the form
            	$$f=\sum_{j\geq 0, j\neq k}f_k,$$
            	where each $f_j$ is a homogeneous polynomial of degree $j$.
            	Hence $(E^k)_t^{\bot}$ as a vector space is independent of the choice of the hermitian metric $h^F$ and is also a holomorphic subbundle of $E$.
            	By Lemma \ref{lem:key lemma} and Lemma \ref{the direct sum of holomorphic vector bundles}, we know Inequality (\ref{curvature inequality}) holds.\\
            	\indent  For any $N\in \mathbb{N},\ s>\frac{1}{N^2}$, set
			$$d(z,\Omega_{t_0}):=\inf_{w\in\Omega_{t_0}}|z-w|,\ \Omega_{t_0,s}^N:=\{z\in\mc^m|\ \frac{1}{N^2}<d(z,\Omega_{t_0})<s\},$$
			Fix an $s$ such that $\Omega_{t_0,s}^N\subset\subset D$ for all $N>>1$.
			It is clear that
            $$\max\nolimits_{(\frac{1}{N^3},\frac{1}{N^3})}\{0,\rho\}=\rho \text{ on }\Omega^N_{t_0,s}$$
            for all $N>>1$, then we have
			\begin{equation}\label{curvature inequality2}
            \int_{\Omega_{t_0,s}}N\sum_{j,l,\beta,\gamma}H(h_{N,\epsilon}^F)_{jl\beta\gamma}u_{j\beta}\overline{u_{l\gamma}}d\lambda\geq
                        \int_{\Omega_{t_0,s}}\frac{N}{2}\sum_{j,l,\beta}H(\rho)_{jl}u_{j\beta}\overline{u_{l\beta}}e^{-N\rho}d\lambda
            \end{equation}
			on $\Omega^N_{t_0,s}$ for all $N>>1$,  where $H(\rho)_{jl}=\rho_{jl}-\sum_{p,q}\rho^{pq}\rho_{jq}\rho_{pl}$. The remaining is similar to the proof of Theorem 4.2 of \cite{DHQ}.
           We will first prove $i\Theta^{(E^k,h_N^E)}\geq \theta$ in the sense of Nakano for some $\theta$, then by Lemma \ref{lem: optimal L2-estimate}, we get $i \Theta_{(E^k, h)} \geq\theta$
           whenever $N\rw \infty$, which will complete the proof. The details is stated as follows.\\
             \indent By the assumption and the tube neighborhood theorem, it is clear that there are  constants $\delta_0, \delta_1>0$ such that (in the following, whenever we say a number is a constant,
             we mean that it is a number that does not depend on the sections of $E$ and $N,\epsilon,s$)
            	\begin{align*}
            		&\int_{\Omega_{t_0,s}}N\sum_{j,l,\beta}H(\rho)_{jl}u_{j\beta}\overline{u_{l\beta}}e^{-N\rho}d\lambda\geq \delta_0\int_{\Omega_{t_0,s}}N\sum_{j,\beta}|u_{j\beta}|^2 e^{-N\rho}d\lambda\\
            		\geq & \delta_1 \int_{\zeta\in\partial \Omega_{t_0}} dS(\zeta)\int^s_{1/N^2} N\sum_{j,\beta}|u_{j\beta}(\zeta+\tau \mathbf{n}_\zeta)|^2 e^{-N\rho(\zeta+\tau \mathbf{n}_\zeta)}d\tau\\
            		\geq & \delta_1 \sum_{j,\beta}\int_{\zeta\in\partial \Omega_{t_0}} dS(\zeta) \inf_{1/N^2\leq \tau\leq s} |u_{j\beta}(\zeta+\tau \mathbf{n}_\zeta)|^2 \int^s_{1/N^2}  N  e^{-N\rho(\zeta+\tau \mathbf{n}_\zeta)}d\tau\\
            		\geq & \delta_1 \sum_{j,\beta}\int_{\zeta\in\partial \Omega_{t_0}} dS(\zeta) \inf_{0\leq \tau\leq s} |u_{j\beta}(\zeta+\tau \mathbf{n}_\zeta)|^2 \int^s_{1/N^2} N  e^{-NT\tau}d\tau,
            	\end{align*}
            	where $\mathbf{n}_\zeta$ is the unit outward normal vector of $\partial\Omega_{t_0}$ at $\zeta$,  and $T>0$ is a constant such that $\rho(\zeta+\tau \mathbf{n}_\zeta)\leq T\tau$ for all $\zeta\in\partial\Omega_{t_0}$ and $0\leq\tau\leq s$.
            	Since
             $$\lim_{N\ra\infty}\int^s_{1/N^2} N  e^{-NT\tau}d\tau=\frac{1}{T}>0,$$
             then  for $N$ sufficiently large, we have
            	\begin{equation}\label{curvature inequality3}
            	\int^s_{1/N^2} N e^{-NT\tau}d\tau\geq \frac{1}{2T}.
            	\end{equation}
            Combining Inequalities (\ref{curvature inequality}), (\ref{curvature inequality2}) and (\ref{curvature inequality3}),
            we know there is a constant $\delta_4>0$ such that, for any sections $u_1, \cdots  , u_n$ of $E^k$, we have
			\begin{equation}\label{curvature inequality}
                \sum_{j,l}\left. h_{N,\epsilon}^{E}\left(\Theta_{jl}^{(E^k, h_{N,\epsilon}^E)} u_j, u_l\right)\right|_{t_0} \geq
                             \delta_4\sum_{j,\beta}\int_{\partial \Omega_{t_0}} \inf_{0\leq \tau\leq s} |u_{j\beta}(\zeta+\tau \mathbf{n}_\zeta)|^2dS(\zeta).
             \end{equation}
            	Let $\epsilon\ra 0+$, we get
            	\begin{align}\label{eq:curvature estimate vir boundary}
            		\sum_{j,l}\left. h^{E}_N\left(\Theta_{jl}^{(E^k, h^E_{N})} u_j, u_l\right)\right|_{t_0}
            		\geq \delta_4\sum_{j,\beta}\int_{\partial\Omega_{t_0}} \inf_{0\leq \tau\leq s} |u_{j\beta}(\zeta+\tau \mathbf{n}_\zeta)|^2 dS
            	\end{align}
            	for $N$ sufficiently large.\\
                \indent We claim that there is a constant $\delta_5>0$ such that, for any sections $u_1, \cdots  , u_n$ of $E^k$, we have
                \begin{equation}\label{curvature inequality4}
                \sum_{j,\beta}\int_{\partial\Omega_{t_0}} \inf_{0\leq \tau\leq s} |u_{j\beta}(\zeta+\tau \mathbf{n}_\zeta)|^2 dS\geq
                  \delta_5\sum_jh^E(u_j,u_j).
                \end{equation}
                Without loss of generality, we only need to sections $u_1, \cdots  , u_n$ of $E^k$ such that
                $$\sum_{j,\beta}\int_{\Omega_{t_0}}|u_{j\beta}|^2d\lambda=1.$$        	
            	Since each $u_{j\beta}$ is a homogenous polynomial of degree $k$ and $\Omega_{t_0}$ contains the origin,
            	we may choose a constant $M_1>0$ and a large ball $B$ with $\overline D\subset B$ such that
            	$\sum_{j,\beta}\int_B|u_{j\beta}|^2d\lambda(z)\leq M_1$. By Cauchy's Inequality for holomorphic functions, there is a constant $M_2>0$ such that
            	$\sum_j|du_j^2|\leq M_2$ on $D$, then it is obvious that
            	\begin{equation}\label{eq:curvature estimate via b}
            		\sum_{j,\beta}\inf_{0\leq \tau\leq s} |u_{j\beta}(\zeta+\tau \mathbf{n}_\zeta)|^2\geq \sum_{j,\beta}|u_{j\beta}(\zeta)|^2-nrsM_2
            	\end{equation}
            	for all $\zeta\in\partial\Omega_{t_0}$.
                By  Corollary 1.7 of \cite{DJQ}, there is a constant $\delta_6>0$ such that
                $$\sum_{j,\beta}\int_{\partial\Omega_{t_0}}|u_{j\beta}(\zeta)|^2dS\geq \delta_6\sum_{j,\beta}\int_{\Omega_{t_0}}|u_{j\beta}|^2d\lambda=\delta_6.$$
            	Choose $N>>1$ enough large and $s>0$ enough small, then by continuity, we know there exists a constant $\delta_7>0$ such that
                $$\sum_{j,\beta}\int_{\Omega_{t_0}}|u_{j\beta}|^2d\lambda-nrs M_2\geq \delta_7\sum_jh^E(u_j,u_j),$$
                which completes the proof of Inequality  (\ref{curvature inequality4}).  \\
            	\indent By Inequality (\ref{eq:curvature estimate vir boundary}) and Inequality
            (\ref{eq:curvature estimate via b}), there is a constant $\delta_8>0$ such that for $N$ sufficiently large and all sections
            $u_1, \cdots  , u_n$ of $E^k$, we have
            \begin{equation}
            		\sum_{j,l}\left. h^{E}_N\left(\Theta_{jl}^{(E^k, h^E_{N})} u_j, u_l\right)\right|_{t_0}
            		\geq \delta_8\sum_{j}h^E(u_j,u_j).
             \end{equation}
             In another word, if we take (where we will shrink $U$ if necessary to gurantee $\delta_8$ is a continuous function of $t_0$)
            	$$\theta:=i \delta_8 \sum_j  dt_j \wedge d\bar{t}_j \otimes \operatorname{Id}_{E^k} \in C^0 (U, \wedge^{1,1}T^*_U \otimes \operatorname{End}(E^k)),$$
            	then we have $i \Theta^{(E^k, h^E_N)} \geq\theta$ in the sense of Nakano.\\
            	\indent Now we prove $i \Theta^{(E^k, h^E)} \geq\theta$ in the sense of Nakano. Let $\psi$ be a strictly plurisubharmonic function on $U$, and $f \in C_c^{\infty} (U, \wedge^{n,1}T^*_U \otimes E^k)$ satisfies $\bar{\partial}f=0$ and
            	$$\int_Uh^E\left(B_{i\partial\bar\partial\psi, \theta}^{-1}f, f\right)e^{-\psi} d\lambda< +\infty,$$
            	where  $B_{i\partial\bar\partial\psi, \theta}$ is given as in Lemma \ref{lem: optimal L2-estimate}.	Then there exists a constant $M_3>0$ such that
            	$$\int_U h_N^E\left(B_{i\partial\bar\partial\psi, \theta}^{-1} f, f\right)e^{-\psi} d\lambda\leq M_3,\ \forall N>>1.$$
            	By Lemma \ref{lem: L2 estimate Nakano}, for all $N>>1$, there are measurable sections $u_N$ of  $\wedge^{n,0}T^*_U \otimes E^k$ on $U$, such that $\bar{\partial}u_N=f$ and
            	$$\int_U h_N^E\left(u_N,u_N\right)e^{-\psi} d\lambda_t \leq \int_Uh_N^E\left(B_{i\partial\bar\partial\psi, \theta}^{-1}f, f\right)e^{-\psi} d\lambda_t \leq M_3.$$
            	Since $h^E_N=h^E$ on $\Omega$, we have
            	$$\int_U h^E(u_N,u_N)e^{-\psi} d\lambda\leq \int_U h_N^E\left(u_N,u_N\right)e^{-\psi}d\lambda\leq M_3,\ \forall N>>1.$$
            	Since every bounded sequence in a Hilbert space has a weakly convergent subsequence, by taking a subsequence if necessary,
            we may assume $u_N$ converges weakly to $u$, which is a measurable section of $\wedge^{n,0}T^*_U \otimes E^k$. Moreover we have  $\bar\partial u=f$ in the sense of distribution and
            $$\int_Uh^E(u,u)e^{-\psi} d\lambda\leq\liminf_{N\ra\infty}\int_U h^E(u_N,u_N)e^{-\psi} d\lambda\leq M_3.$$
            By Lebesgue's Dominated Convergence Theorem, we know
            	$$\lim_{N\ra\infty}\int_U h_N^E\left(B_{i\partial\bar\partial\psi, \theta}^{-1} f, f\right)e^{-\psi} d\lambda=
            	\int_U h^E\left(B_{i\partial\bar\partial\psi, \theta}^{-1}f, f\right)e^{-\psi} d\lambda,$$
            	so we get
            	$$\int_Uh^E(u,u)e^{-\psi} d\lambda\leq\int_U h^E\left(B_{i\partial\bar\partial\psi, \theta}^{-1}f, f\right)e^{-\psi} d\lambda.$$
            	It follows from Lemma \ref{lem: optimal L2-estimate} that $i \Theta^{(E^k, h)} \geq \theta$ in the sense of Nakano.
            	In particular, the curvature of $(E^k,h)$ is strictly positive in the sense of Nakano.
            \end{proof}
            
            \begin{rem}
            Assume that $D\subset\mc^m$ is a bounded circular domain containing the origin and $\varphi$ is an $\mathbb{S}^1$-invariant continuous function on $\overline D$.
Let $A^2(D,\varphi)$ be the space of $L^2$ holomorphic functions on $D$ with weight $e^{-\varphi}$.
Then ${\rm P}^k\subset A^2(D,\varphi)$ for $k\geq 0$, and ${\rm P}^k\bot {\rm P}^l$ provided that $k\neq l$.
Moreover, we have the decomposition
$$A^2(D,\varphi)=\overline{\mathop{\bigoplus}_k} {\rm P}^k,$$
where $\overline\bigoplus$ means taking orthogonal direct sum first and then taking closure.
In terms of representation theory, $\mathbb{S}^1$ has a natural unitary representation on $A^2(D,\varphi)$
and $\rm P^k$ is the submodule corresponding to the irreducible character $\alpha\mapsto\alpha^k$ of $\mathbb{S}^1$.
For Reinhardt domains, we have similar decompositions by considering representation of the torus group $\mathbb T^m$.
In this meaning, Theorem \ref{thm(extend):curvartue positive circular domains} and Theorem \ref{qin:application} can be generalized to any compact group actions.
The case of the trivial group action (without group action) is carried out in detail by the last author \cite{qin}.
            \end{rem}

            \section{The proof of Theorem \ref{qin:convex}}
            In this section, we give the proof of Theorem \ref{qin:convex}.
            Parallel to Definition \ref{def:strict p.s.c family}, we give the following
            
\begin{defn}\label{def:strict convex family}
			Let $U\subset\mr^n$ be a domain, and let $p\colon \mr^n\times\mr^m\rw \mr^n$ be the natural projection.\bi
			\item[(1)] A \emph{family of bounded domains} (of dimension $m$ over $U$) is a domain $D\subset U\times\mr^m$ such that $p(D)=U$
             and  all fibers $D_t:=p^{-1}(t)(t\in U)$ are bounded.
			\item[(2)] A family of bounded domains $D$ over $U$ has\emph{ $C^2$  boundary }if  there exists a $C^2$-function $\rho$
			defined on $U\times\mr^m$ such that $D=\{(t,x)\in U\times\mr^m|\ \rho(t,x)<0\}$ and $d(\rho|_{D_t})\neq 0$
            on $\partial D_t$ for all $t\in U$. Such a function $\rho$ is  called a \emph{defining function} of $D$.
			\item[(3)] A family of bounded domains $D\subset U\times \mc^m$ is called \emph{strictly convex} if it admits a $C^2$-boundary defining function that is
            strictly convex on some neighborhood of $\overline{D}$.
			\ei
		\end{defn}

The proof of Theorem \ref{qin:convex} is as follows.
           \begin{proof}
            Consider the map
            $$\phi\colon \mc_z^m\rw (\mc_w^*)^m,\ (z_1,\cdots,z_m)\mapsto (e^{z_1},\cdots,e^{z_m}).$$
            Similar as in the proof of Theorem \ref{thm:Ber}, we may get a measure on $(\mc_w^*)^n$ via $\phi$ which satisfies
            $$
            d\mu(w)=\frac{1}{|w_1|^2\cdots|w_n|^2}d\lambda(w),\ \forall w\in (\mc^*)^n.
            $$
            Set $U:=U_0+i\mr^n\subset \mc_\tau^n$. For any $\tau\in U$, set
            $$\Omega_\tau:=D_{\re(\tau)}+i\mr^m\subset \mc_z^m,\ (\Omega')_\tau:=\phi(D_{\re(\tau)}+i \mr^m)\subset (\mc_w^*)^m,$$
            then both $\Omega:=\cup_{\tau\in U}\Omega_\tau$ and $\Omega':=\cup_{\tau\in U}(\Omega')_\tau$ are strictly pseudoconvex family of domains over $U$.
            It is clear that each $\Omega_\tau$ is a (connected) Reinhardt domain for each $\tau\in U$. We extend $g^E$  to a Hermitian metric $h^{E'}$  on the trivial vector bundle $E':=U\times\mc^r$ over $U$
            such that $h^{E'}$ is independent of $\im (\tau)$ for any $\tau\in U$. We also extend $g^F$ to a Hermitian metric $h'$ on the trivial vector bundle $F':=\Omega\times \mc^r$ over $\Omega$
            such that $h'$ is independent of $\im(\tau),\ \im(z)$ for any $(\tau,z)\in\Omega.$
             $h''$ clearly induces a Hermitian metric $h''$  on the trivial vector bundle $F'':=\Omega'\times \mc^r$. It is obvious that $(F'',h'')$ is Nakano positive on  some neighborhood of the closure of $\Omega'$ by assumption.
             Let
             $$h_{(\tau,w)}:=\frac{h''_{(\tau,w)}}{|w_1|^2\cdots |w_m|^2}
              \text{ for any }w\in \Omega_\tau \text{ and any }\tau\in U,$$
            then $(F'',h)$ is again Nakano positive on some neighborhood of the closure of $ \Omega'$, and $h_{(\tau,w)}$ is $\mathbb{T}^n$-invariant with respect to $w$ for any $w\in \Omega_\tau$ and any $\tau\in U$. Let $\{e_1'',\cdots,e_r''\}$ be the canonical holomorphic frame of $F''$.\\
            \indent For any $\tau\in U$, by Lemma \ref{measure:1}, we have
             $$h_\tau^{E'}(u,v)=\frac{1}{(2\pi)^m}\int_{(\Omega')_\tau}\sum_{\lambda,\mu=1}^ru_\lambda \overline{v_\mu}h_{(\tau,w)}(e_\lambda'',e_\mu'')d\lambda(w),\ \forall u,v\in (E')_\tau.$$
            By Theorem \ref{qin:application}, we know $(E',h^{E'})$ is Nakano strictly positive on $U$. Therefore, $(E,g^E)$ is also strictly positive in the sense of Nakano.
           \end{proof}

	\end{document}